\documentclass{amsproc}
\usepackage{amsfonts}

\theoremstyle{plain}

\newtheorem{proposition}{Proposition}

\newtheorem{theorem}{Theorem}
\numberwithin{equation}{section}

\begin{document}
\title[A locally compact transitive groupoid without open range map ]{A
second countable locally compact transitive groupoid without open range map}
\author{M\u{a}d\u{a}lina Roxana Buneci}
\address{University Constantin Br\^{a}ncu\c{s}i of T\^{a}rgu-Jiu, Calea
Eroilor No. 30, 210135 T\^{a}rgu-Jiu, Rom\^{a}nia.}
\email{mbuneci@yahoo.com}
\urladdr{http://www.utgjiu.ro/math/mbuneci/}
\thanks{}
\date{}
\dedicatory{}
\thanks{}

\begin{abstract}
Dana P. Williams raised in [Proc. Am. Math. Soc., Ser. B, 2016]
the following question: \ Must a second countable, locally compact,
transitive groupoid have open range map? This paper gives a negative answer
to that question. Although a second countable, locally compact transitive
groupoid $G$ may fail to have open range map, we prove that we can replace
its topology with a topology which is also second countable, locally
compact, and with respect to which $G$ is a topological groupoid whose range
map is open. Moreover, the two topologies generate the same Borel structure
and coincide on the fibres of $G$..
\end{abstract}

\maketitle

\section{Introduction}

In order to construct convolution algebras associated to a locally compact
topological groupoid one needs an analogue of the Haar measure on a locally
compact group. Starting with the work of Jean Renault \cite{re1}, this
analogue is a system of measures, called Haar system, subject to suitable
support, invariance and continuity conditions. According to a result of
Anthony Seda \cite{se3}, the continuity assumption is essential for the
Renault's construction \cite{re1} of the $C^{\ast }$-algebra associated to a
locally compact groupoid. This continuity assumption entails that the range
map is open (\cite[Proposition I.4]{we} or \cite[p. 118]{se3}). \ As
Williams pointed out \cite[Question 3.5]{wi}, while there certainly exist
groupoids that fail to have open range maps, most of these examples  are
group bundles which are as far as from been principal groupoids or
transitive groupoids as possible. This led him to the next question \cite[%
Question 3.5]{wi}: must a second countable, locally compact, transitive
groupoid have open range map? \ In this paper we construct a second
countable, locally compact, transitive and principal groupoid that fails to
have open range map (and hence open domain map). \ 

In addition, we prove that for every second countable, locally compact
topology $\tau $ on a transitive groupoid $G$ making $G$ a topological
groupoid, there is a topology $\tilde{\tau}$ which is also second countable,
locally compact, and with respect to which $G$ is a topological groupoid
with open range map. Moreover, $\tau $ and $\tilde{\tau}$ generate the same
Borel structure and coincide on the fibres of $G$.

We shall use the definition of a topological groupoid given by Jean Renault
in \cite{re1}. For a groupoid $G$, $G^{\left( 2\right) }$ will denote the
set of the composable pairs and $G^{\left( 0\right) }$ its unit space. \ As
usual the inverse map will be written $x\rightarrow x^{-1}\,\left[
:G\rightarrow G\right] $ \ and the product map will be written $\left(
x,y\right) \rightarrow xy\,\left[ :G^{\left( 2\right) }\rightarrow G\right] $%
. For each $x\in G$, $r\left( x\right) =xx^{-1}$, respectively, $d\left(
x\right) =x^{-1}x$ will denote the range, respectively the domain (source)
of $x$ in $G^{\left( 0\right) }$ (thus $r:G\rightarrow G^{\left( 0\right) }$%
, respectively $d:G\rightarrow G^{\left( 0\right) }$ will be the range map,
respectively the domain/source map). For each $u\in G^{\left( 0\right) }$,
the fibre of the range, respectively domain map over $u$ is denoted $%
G^{u}=r^{-1}\left( \left\{ u\right\} \right) $, respectively, $%
G_{u}=d^{-1}\left( \left\{ u\right\} \right) $.

A groupoid $G$ is said to be transitive if for every $u$,$v\in G^{\left(
0\right) }$, there is $x\in G$ such that $r\left( x\right) =u$ and $d\left(
x\right) =v$. A groupoid is called principal if the map $\left( r,d\right)
:G\rightarrow G^{\left( 0\right) }\times G^{\left( 0\right) }$, defined by $%
\left( r,d\right) \left( x\right) =\left( r\left( x\right) ,d\left( x\right)
\right) $ for all $x\in G$, is injective.

If $A$, $B$ $\subset G$ and $x\in G$, one may form the following subsets of $%
G$:%
\begin{eqnarray*}
A^{-1} &=&\left\{ x\in G:\;x^{-1}\in A\right\} \\
AB &=&\left\{ xy:\left( x,y\right) \in G^{\left( 2\right) }\cap \left(
A\times B\right) \right\} \\
xA &=&\{x\}A\text{ and }Ax=A\left\{ x\right\} \text{.}
\end{eqnarray*}

A topological groupoid consists of a groupoid $G$ and a topology compatible
with the groupoid structure. This means that the inverse map $x\mapsto
x^{-1}\;\left[ :G\rightarrow G\right] $ is continuous, as well as the
product map $\left( x,y\right) \mapsto xy\;\left[ :G^{\left( 2\right)
}\rightarrow G\right] $ is continuous, where $G^{\left( 2\right) }$ has the
induced topology induced from $G\times G$. By a locally compact groupoid we
mean a topological groupoid whose topology is locally compact (Hausdorff)

\section{A second countable locally compact transitive groupoid without open
range map\label{tno}}

Let us modify the usual topology of the space of real numbers $\mathbb{R}$
in the points of the form $\frac{3}{2^{n+2}}$ and $\frac{5}{2^{n+2}}$ with
 $n\in \mathbb{N}$. For every $x\in \mathbb{R}$, let 
\begin{equation*}
\mathcal{B}_{x}=\left\{ 
\begin{array}{l}
\left\{ \left[ \frac{3}{2^{n+2}},\frac{3}{2^{n+2}}+\varepsilon \right) \text{%
, }\varepsilon >0\right\} \text{, if }x=\frac{3}{2^{n+2}}\text{ (}n\in 
\mathbb{N}\text{)} \\ 
\left\{ \left( \frac{5}{2^{n+2}}-\varepsilon ,\frac{5}{2^{n+2}}\right] \text{%
, }\varepsilon >0\right\} \text{, if }x=\frac{5}{2^{n+2}}\text{ (}n\in 
\mathbb{N}\text{)} \\ 
\left\{ \left( x-\varepsilon ,x+\varepsilon \right) \text{, }\varepsilon
>0\right\} \text{, if }x\in \mathbb{R}\setminus \bigcup\limits_{n=0}^{\infty
}\left\{ \frac{3}{2^{n+2}},\frac{5}{2^{n+2}}\right\} .%
\end{array}%
\right. 
\end{equation*}%
and%
\begin{equation*}
\mathcal{F}_{x}=\left\{ V\subset \mathbb{R}:\text{there is }U\in \mathcal{B}%
_{x}\text{ such that }U\subset V\right\} 
\end{equation*}%
Let 
\begin{equation*}
\tau _{0}=\left\{ O\subset \mathbb{R}\text{: if }x\in O\text{, then }O\in 
\mathcal{F}_{x}\right\} 
\end{equation*}%
be the unique topology on $X=\mathbb{R}$ with the property that for every $%
x\in X$, $\ \mathcal{F}_{x}$ is the family of neighborhoods of $x$. Then $%
\left( X,\tau _{0}\right) $ is a Hausdorff second countable topological
space. Moreover if $x\in \mathbb{R}\setminus \left(
\bigcup\limits_{n=0}^{\infty }\left\{ \frac{3}{2^{n+2}},\frac{5}{2^{n+2}}%
\right\} \cup \left\{ 0\right\} \right) $ and $\varepsilon >0$ is small
enough such that $\left[ x-\varepsilon ,x+\varepsilon \right] \cap
\bigcup\limits_{n=0}^{\infty }\left\{ \frac{3}{2^{n+2}},\frac{5}{2^{n+2}}%
\right\} =\emptyset $, then $\left[ x-\varepsilon ,x+\varepsilon \right] $
is a compact neighborhood of $x$ with respect to the topology $\tau _{0}$.
Also if $n\in \mathbb{N}$ and $0<\varepsilon <\frac{1}{2^{n+1}}$, then $%
\left[ \frac{3}{2^{n+2}},\frac{3}{2^{n+2}}+\varepsilon \right] $ is a
compact neighborhood of $\frac{3}{2^{n+2}}$\ and $\left[ \frac{5}{2^{n+2}}%
-\varepsilon ,\frac{5}{2^{n+2}}\right] $ is a compact neighborhood of $\frac{%
5}{2^{n+2}}$ .

Let $G=X\times X=\mathbb{R}\times \mathbb{R}$ be the pair groupoid (product: 
$\left( x,y\right) \left( y,z\right) =\left( x,z\right) $, inverse: $\left(
x,y\right) ^{-1}=\left( y,x\right) $). For every $\left( x,y\right) \in G$
let 
\begin{equation*}
\mathcal{B}_{\left( x,y\right) }=\left\{ 
\begin{array}{l}
\left\{ A\times B:\left( x,y\right) \in A\times B\text{,}~A,B\in \tau
_{0}\right\} \text{, if }x\neq 0\text{ and }y\neq 0 \\ 
\left\{ \left\{ 0\right\} \times B:y\in B\text{, }~B\in \tau _{0}\right\} 
\text{, if }x=0\text{ and }y\neq 0 \\ 
\left\{ A\times \left\{ 0\right\} :x\in A\text{, }~A\in \tau _{0}\right\} 
\text{, if }x\neq 0\text{ and }y=0 \\ 
\left\{ \left\{ \left( 0,0\right) \right\} \cup U_{n}:n\in \mathbb{N}%
\right\} \text{, if }\left( x,y\right) =\left( 0,0\right) \text{,} \\ 
\text{where }U_{n}=\bigcup\limits_{k=n}^{\infty }\left[ \frac{3}{2^{k+2}},%
\frac{5}{2^{k+2}}\right] \times \left[ \frac{3}{2^{k+2}},\frac{5}{2^{k+2}}%
\right] \text{ for all }n\in \mathbb{N}\text{.}%
\end{array}%
\right. 
\end{equation*}%
and%
\begin{equation*}
\mathcal{F}_{\left( x,y\right) }=\left\{ V\subset \mathbb{R}\times \mathbb{R}%
:\text{there is }U\in \mathcal{B}_{\left( x,y\right) }\text{ such that }%
U\subset V\right\} \text{.}
\end{equation*}%
Let us endow $G=\mathbb{R}\times \mathbb{R}$ with the unique topology $\tau
_{1}$ with the property that for every $\left( x,y\right) \in G$, $\ 
\mathcal{F}_{\left( x,y\right) }$ is the family of neighborhoods of $\left(
x,y\right) $. It is easy to see that the inverse map is continuous with
respect to $\tau _{1}$. Since for all subsets $A,B,C\subset X=\mathbb{R}$,

\begin{equation*}
\left( A\times C\right) \left( C\times B\right) \subset A\times B
\end{equation*}%
it follows that the product map is continuous in the points of the form $%
\left( \left( x,y\right) ,\left( y,z\right) \right) \in G^{\left( 2\right) }$
with $x\neq 0$ and $z\neq 0$. For every $n\in \mathbb{N}$, let $%
U_{n}=\bigcup\limits_{k=n}^{\infty }\left[ \frac{3}{2^{k+2}},\frac{5}{2^{k+2}%
}\right] \times \left[ \frac{3}{2^{k+2}},\frac{5}{2^{k+2}}\right] $. The
continuity of the product map in $\left( \left( 0,0\right) ,\left(
0,0\right) \right) \,$\ is the consequence of the fact that 
\begin{equation*}
\left( \left\{ \left( 0,0\right) \right\} \cup U_{n}\right) \left( \left\{
\left( 0,0\right) \right\} \cup U_{n}\right) =\left\{ \left( 0,0\right)
\right\} \cup U_{n}\text{.}
\end{equation*}%
Since for all subsets $B\subset X=\mathbb{R}$,%
\begin{equation*}
\left( \left\{ 0\right\} \times B\right) \left( B\times \left\{ 0\right\}
\right) =\left\{ \left( 0,0\right) \right\} \subset \left\{ \left(
0,0\right) \right\} \cup U_{n}
\end{equation*}%
it follows that the product map is continuous is continuous in the points of
the form $\left( \left( 0,y\right) ,\left( y,0\right) \right) $ such that $%
y\neq 0$. The fact that for all subsets $B\subset X=\mathbb{R}$,%
\begin{eqnarray*}
\left( \left\{ \left( 0,0\right) \right\} \cup U_{n}\right) \left( \left\{
0\right\} \times B\right) &=&\left\{ \left( 0,0\right) \right\} \left(
\left\{ 0\right\} \times B\right) \\
&=&\left\{ 0\right\} \times B
\end{eqnarray*}%
implies \ that the product map is continuous in the points of the form $%
\left( \left( 0,0\right) ,\left( 0,y\right) \right) $ such that $y\neq 0$.
Similarly, the product map is continuous in the points of the form $\left(
\left( y,0\right) ,\left( 0,0\right) \right) $ such that $y\neq 0$.
Therefore $\left( G,\tau _{1}\right) $ is a topological groupoid.

For every $x\in \mathbb{R}\setminus \left\{ 0\right\} $, let $K_{x}$ be a
compact neighborhood of $x$ with respect to $\tau _{0}$. Then $K_{x}\times
K_{y}$ is a compact neighborhood of $\left( x,y\right) $ with respect to the
topology $\tau _{1}$. For every $x\in \mathbb{R}\setminus \left\{ 0\right\} $%
, $\left\{ 0\right\} \times K_{x}$, respectively $K_{x}\times \left\{
0\right\} $ is a compact neighborhood of $\left( 0,x\right) $, respectively $%
\left( x,0\right) $ with respect to $\tau _{1}$. Let us prove that for each $%
m\in \mathbb{N}$, 
\begin{equation*}
K_{m}=\left\{ \left( 0,0\right) \right\} \bigcup \left(
\bigcup\limits_{k=m}^{\infty }\left[ \frac{3}{2^{k+2}},\frac{5}{2^{k+2}}%
\right] \times \left[ \frac{3}{2^{k+2}},\frac{5}{2^{k+2}}\right] \right) 
\end{equation*}%
is a compact neighborhood of $\left( 0,0\right) $. Let $\left( \left(
x_{n},y_{n}\right) \right) _{n}\in K_{m}$. If 
\begin{equation*}
\left\{ n\in \mathbb{N}:\left( x_{n},y_{n}\right) =\left( 0,0\right)
\right\} 
\end{equation*}%
is infinite, then $\left( \left( x_{n},y_{n}\right) \right) _{n}$ has a
subsequence converging to \thinspace $\left( 0,0\right) \in K_{m}$. If $%
\left\{ n\in \mathbb{N}:\left( x_{n},y_{n}\right) =\left( 0,0\right)
\right\} $ is finite, then there is an integer $n_{0}\geq m$ such that $%
\left( x_{n},y_{n}\right) \in \bigcup\limits_{k=n}^{\infty }\left[ \frac{3}{%
2^{k+2}},\frac{5}{2^{k+2}}\right] \times \left[ \frac{3}{2^{k+2}},\frac{5}{%
2^{k+2}}\right] $ for all $\ n\geq n_{0}$. In this case for every $n\geq
n_{0}$, there are $u_{n}$, $v_{n}\in \left[ 3,5\right] $ and $k_{n}\in 
\mathbb{N}$, such that $x_{n}=\frac{u_{n}}{2^{k_{n}+2}}$ and $y_{n}=\frac{%
v_{n}}{2^{k_{n}+2}}$. If $\left( k_{n}\right) _{n}$ is unbounded, then $%
\left( k_{n}\right) _{n}$ has a subsequence that diverges to $\infty $ and
hence $\left( x_{n}\right) _{n}$ and $\left( y_{n}\right) _{n}$ have
subsequences which converge to $0$ in the usual topology on $\mathbb{R}$.
Since for all $n$, 
\begin{multline*}
K_{m}\cap \left( \left( -\frac{5}{2^{n+2}},\frac{5}{2^{n+2}}\right) \times
\left( -\frac{5}{2^{n+2}},\frac{5}{2^{n+2}}\right) \right) \subset  \\
\subset \left\{ \left( 0,0\right) \right\} \bigcup
\bigcup\limits_{k=n}^{\infty }\left[ \frac{3}{2^{k+2}},\frac{5}{2^{k+2}}%
\right] \times \left[ \frac{3}{2^{k+2}},\frac{5}{2^{k+2}}\right] \text{,}
\end{multline*}%
it follows that $\left( \left( x_{n},y_{n}\right) \right) _{n}$ has a
subsequence converging to $\left( 0,0\right) $ with respect to $\tau _{1}$.
If $\left( k_{n}\right) _{n}$ is bounded, then it has a convergent
subsequence in the usual topology on $\mathbb{R}$, or equivalently a
stationary subsequence. Also $\left( u_{n}\right) _{n}$ and $\left(
v_{n}\right) _{n}$ have convergent subsequences in the usual topology on $%
\mathbb{R}$. Thus $\left( x_{n}\right) _{n}$ (respectively, $\left(
y_{n}\right) _{n}$) has a subsequence converging to $\frac{u}{2^{p+2}}$%
(respectively, $\frac{v}{2^{p+2}}$) in the usual topology on $\mathbb{R}$.
Since $u,v\in \left[ 3,5\right] $ , it follows that every neighborhood of $%
\left( \frac{u}{2^{p+2}},\frac{v}{2^{p+2}}\right) $ with respect to the
topology $\tau _{1}$ contains a set of the form $A\times B$, with $A$, $B\in
\tau _{0}$ and $\left( \frac{u}{2^{p+2}},\frac{v}{2^{p+2}}\right) \in $ $%
A\times B$. \ Therefore $\left( \left( x_{n},y_{n}\right) \right) _{n}$ has
a subsequence converging to $\left( \frac{u}{2^{p}},\frac{v}{2^{p}}\right) $
in the topology $\tau _{1}$ of $G=X\times X=\mathbb{R}\times \mathbb{R}$.
Thus $K_{m}$ is a compact neighborhood of $\left( 0,0\right) $.

Therefore $\left( G,\tau _{1}\right) $ is a second countable locally compact
groupoid. Since for any $B\in \tau _{0}$, 
\begin{equation*}
r\left( \left\{ 0\right\} \times B\right) =\left\{ \left( 0,0\right)
\right\} 
\end{equation*}%
is not open in $G^{\left( 0\right) }=\left\{ \left( x,x\right) :x\in
X\right\} $, it follows that the range map is not open (and hence the domain
map $d$ is not open).

\section{Replacing the topology of a locally compact transitive groupoid with a locally transitive topology}

We prove that for every second countable, locally compact, transitive
groupoid $G$ there is a second countable, locally compact topology $\tilde{%
\tau}$ making $G$ a topological groupoid with open range map. Moreover, the
original topology on $G$ and $\tilde{\tau}$ generate the same Borel
structure and coincide on the $r$-fibres and $d$-fibres of $G$. The topology 
$\tilde{\tau}$ is in fact the topology $\tau _{\mathcal{W}}$ introduced in 
\cite[Definition 3.1]{bu4}, where $\mathcal{W}$ is a suitable $G$%
-uniformity. Let us recall that a $G$-uniformity (in the sense of \cite[%
Definition 2.1]{bu4}) is a collection $\left\{ W\right\} _{W\in \mathcal{W}}$%
\ of subsets of a groupoid $G$ satisfying the following conditions:

\begin{enumerate}
\item $G^{\left( 0\right) }\subset W\subset G$ for all $W\in \mathcal{W}$.

\item If $W_{1}$, $W_{2}\in \mathcal{W}$, then there is $W_{3}\subset
W_{1}\cap W_{2}$ such that $W_{3}\in \mathcal{W}$.

\item For every $W_{1}\in \mathcal{W}$ there is $W_{2}\in \mathcal{W}$ such
that $W_{2}W_{2}\subset W_{1}$.

\item $W=W^{-1}$for all $W\in \mathcal{W}$.
\end{enumerate}

If $G$ is groupoid endowed with a topology, then a $G$-uniformity $\mathcal{W%
}$ is said to be compatible with the topology of the $r$-fibres (in the
sense of \cite[Definition 3.4]{bu4}) if for every $u\in G^{\left( 0\right) }$
and every open neighborhood $U$ of $u$, there is $W\in \mathcal{W}$ such
that $W\cap G^{u}\subset U\cap G^{u}$ and $u$ is in the interior of $W\cap
G^{u}$ with respect to the topology on $G^{u}$ coming from $\left( G,\tau
\right) .$

A subset $K$ of $G$ is diagonally compact (in the sense of \cite[p. 10]{mw})
if $K\cap r^{-1}\left( L\right) $ and $K\cap d^{-1}\left( L\right) $ are
compact whenever $L$ is a compact subset of $G^{\left( 0\right) }$.

\begin{proposition}
\label{Wn}If $G$ is a second countable, locally compact Hausdorff groupoid,
then $G$ admits a countable $G$-uniformity $\left\{ W_{n}\right\} _{n\in 
\mathbb{N}}$ compatible with the topology of the $r$-fibres such that for
every $n\in \mathbb{N}$, $W_{n}$ is a \ diagonally compact neighborhood of $%
G^{\left( 0\right) }$
\end{proposition}

\begin{proof}
Since $G$ is a second countable, locally compact Hausdorff space, it follows
that $G$ is metrizable. Let us denote the metric by $d$. Also since $G$ is a
second countable, locally compact Hausdorff space

it follows that $G$ as well as $G^{\left( 0\right) }$ are paracompact
spaces. Thus $G$ has a fundamental system of diagonally compact
neighborhoods of $G^{(0)}$ (by \cite[Lemma 2.10/p. 10]{mw} or the proof \ of 
\cite[Proposition 1.9]{re1}). For each $n\in \mathbb{N}\setminus \left\{
0\right\} $ let us write%
\begin{equation*}
D_{n}=\left\{ x\in G:d\left( x,r\left( x\right) \right) <\frac{1}{n+1}%
\right\} \text{.}
\end{equation*}

Let $W_{0}$ be a diagonally compact symmetric neighborhood of $G^{(0)}$ such
that $W_{0}\subset D_{0}$. Inductively we construct a $G$-uniformity $%
\left\{ W_{n}\right\} _{n\in \mathbb{N}}$\ consisting in diagonally compact
symmetric neighborhoods of $G^{\left( 0\right) }$ such that $W_{n}\subset
D_{n}$ for all $n\in \mathbb{N}$. Suppose a symmetric neighborhood $W_{n}$
of $G^{\left( 0\right) }$ has already been built. \ Let $V_{n}$ be a
diagonally compact neighborhood of $G^{(0)}$ such that $V_{n}\subset
D_{n+1}\cap W_{n}$. Since $G$ is paracompact, according to \cite[p. 361-362]%
{ra2}, there is a neighborhood $U_{n}$ of $G^{(0)}$ such that $%
U_{n}U_{n}\subset V_{n}$. Let $W_{n+1}$ be a diagonally compact neighborhood
of $G^{(0)}$ such that $W_{n+1}\subset U_{n}$. \ Replacing $W_{n+1}$ with $%
W_{n+1}\cap W_{n+1}^{-1}$, we may assume that $W_{n+1}=W_{n+1}^{-1}$. Thus
we obtain a diagonally compact symmetric neighborhood of $G^{(0)}$ such that%
\begin{equation*}
W_{n+1}\subset W_{n+1}W_{n+1}\subset U_{n}U_{n}\subset V_{n}\subset
D_{n+1}\cap W_{n}\text{.}
\end{equation*}

Let us remark that for every $u\in G^{\left( 0\right) }$ we have%
\begin{equation*}
W_{n}\cap G^{u}\subset D_{n}\cap G^{u}=\left\{ x\in G^{u}:d\left( x,u\right)
\,<\frac{1}{n+1}\right\}
\end{equation*}

for all $n\in \mathbb{N}$. \ Consequently, $\left\{ W_{n}\right\} _{n\in 
\mathbb{N}}$\ is compatible with the topology of the $r$-fibres.
\end{proof}

Let us recall that $A\subset G$ is open with respect to the topology $\tau _{%
\mathcal{W}}$ \cite[Definition 3.1]{bu4} (respectively, $\tau _{\mathcal{W}%
}^{r}$ \cite[p. 59]{bu5}) associated to a $G$-uniformity $\mathcal{W}$ \cite[%
Definition 2.1]{bu4} if and only if for every $x\in A$ there is $W_{x}\in 
\mathcal{W}$ such that $W_{x}xW_{x}\subset A$ (respectively, $xW_{x}\subset A
$). If $G$ is a topological groupoid and if $\mathcal{W}$ is a $G$%
-uniformity compatible with the topology of the $r$- fibres, then by \cite[%
Proposition 3.6]{bu4} \ for every $W_{1}\in \mathcal{W}$ and $x\in G$ there
is $W_{2}\in \mathcal{W}$ such that $W_{2}\cap G_{d\left( x\right)
}^{d\left( x\right) }\subset x^{-1}W_{1}x$ $\ $and by \cite[Proposition 3.7]%
{bu4}, $G$ endowed with $\tau _{\mathcal{W}}$ is a topological groupoid.
Moreover the topologies induced by $\tau _{\mathcal{W}}^{r}$ and $\tau _{%
\mathcal{W}}$ on $r$-fibres coincide. According \cite[Proposition 3.8]{bu4},
the compatibility of the $G$-uniformity $\mathcal{W}$ with the topology of $G
$ ensures that the topologies induced by $\tau _{\mathcal{W}}$ and the
original topology of $G$ on the $r$-fibres $G^{u}$ coincide. \ Thus \ for
each $u\in G^{\left( 0\right) }$ and each $x\in G^{u}$, $\left\{ xW\right\}
_{W\in \mathcal{W}}$ \ is a neighborhood basis (local basis) for $x$ with
respect to the topology induced from $G$ on $G^{u}$.

\begin{proposition}
\label{dens}Let $G$ be a topological transitive groupoid endowed with a $G$%
-uniformity $\mathcal{W}$ compatible with the topology of the $r$-fibres and
let $u\in G^{\left( 0\right) }$. If $\ S\subset G^{u}$ is a dense subset of $%
G^{u}$ with respect to the topology induced from $G$, then $S^{-1}S$ is a
dense subset of $\ G$ with respect to the topology $\tau _{\mathcal{W}}$
\end{proposition}

\begin{proof}
Let $W\in \mathcal{W}$ and $y\in G$. Since $G$ is transitive, it follows
that there is $x\in G^{u}$ such that $d\left( x\right) =r\left( y\right) $.
The density of $S$ implies that there is $s\in S$ such that $s\in xW$ or
equivalently, $x\in sW^{-1}=sW$. Using again the density of $S$ in $G^{u}$
and the fact that $xy\in G^{u}$, \ it follows that there is $t\in S$ such
that $t\in xyW$. Consequently, $xy\in tW^{-1}=tW$. Therefore $y\in
x^{--1}tW\subset Ws^{-1}tW$ or equivalently, $s^{-1}t\in WyW$. Thus any
neighborhood (with respect to $\tau _{\mathcal{W}}$) of $y$ contains at
least one point from $S^{-1}S$.
\end{proof}

A topological groupoid is said to be locally transitive \cite[p. 119]{se3}
(or groupo\"{\i}de microtransitif \cite{eh}) if for every $u\in G^{\left(
0\right) }$ the map $r_{u}$ is open, where $r_{u}:G_{u}\rightarrow G^{\left(
0\right) }$ is defined by $r_{u}\left( x\right) =r\left( x\right) $ for all $%
x\in G_{u}$ and $G_{u}$ is endowed with the topology coming from $G$. Hence
the maps $d_{u}$ are open, where $d_{u}:G^{u}\rightarrow G^{\left( 0\right) }
$, $d_{u}\left( x\right) =d\left( x\right) $ for all $x\in G^{u}$. \
Obviously, every locally transitive groupoid has open range and domain maps.
Conversely, according \cite[Theorem 2.2 A amd Theorem 2.2 N]{mrw} or \cite[%
Theorem 3.2]{ra2}, if $G$ is a second countable, locally compact, transitive
groupoid with open range map, then $G$ is locally transitive. The example
constructed in Section \ref{tno} demonstrates that there are second
countable, locally compact, transitive groupoids which are not locally
transitive. The following theorem shows that if $G$ is a second countable,
locally compact, transitive groupoid $G$, then we can eventually replace the
topology of $G$ with a second countable, locally compact topology $\tilde{%
\tau}$ making $G$ a locally transitive groupoid. In addition, the original
topology on $G$ ant $\tilde{\tau}$ generate the same Borel structure and
coincide on the $r$-fibres (hence on $d$-fibres) of $G$.

\begin{theorem}
Let $G$ be a transitive groupoid endowed with second countable locally
compact Hausdorff \ topology $\tau $ making $G$ a topological groupoid. Then
the topology $\tau $ of $G$ can be replaced with a topology $\tilde{\tau}$
such that:

\begin{enumerate}
\item $G$ is a (topological) locally transitive groupoid with respect to the
topology $\tilde{\tau}$ (hence $G$ has open range map with respect to $%
\tilde{\tau}$).

\item The topology $\tilde{\tau}$ is in general finer than $\tau $. However $%
\tau $ and $\tilde{\tau}$ coincide if  $\left( G,\tau \right) $ is locally
transitive (i.e. for every $u\in G^{\left( 0\right) }$, $r|_{G_{u}}$ is open
with respect to the topology induced by $\tau $ on $G_{u}$).

\item The topologies induced by $\tau $ and $\tilde{\tau}$ on $r$-fibres
(respectively, on $d$-fibres) of $G$ coincide.

\item The topology $\tilde{\tau}$ is second countable and locally compact
Hausdorff.

\item The topologies $\tau $ and $\tilde{\tau}$ generate the same Borel
structure on $G$ (the Borel sets of a topological space are taken to be the $%
\sigma $-algebra generated by the open sets).
\end{enumerate}
\end{theorem}

\begin{proof}
According Proposition \ref{Wn}, $G$ admits a countable $G$-uniformity $%
\mathcal{W}=\left\{ W_{n}\right\} _{n\in \mathbb{N}}$ compatible with the
topology of the $r$-fibres such that for every $n\in \mathbb{N}$, $W_{n}$ is
a diagonally compact neighborhood of $G^{\left( 0\right) }$. Let $\tilde{\tau%
}=\tau _{\mathcal{W}}$. \ 

Then \cite[Proposition 3.6]{bu4} and \cite[Proposition 3.7]{bu4} show $1$
and \cite[Proposition 3.8]{bu4} implies $2$ and $3$.

Let us fix $u\in G^{\left( 0\right) }$ and let $\left\{ x_{n},~n\in \mathbb{N%
}\right\} $ be a dense subset of $G^{u}\,$.

$4$. For each $n$, let $\overset{\circ }{W_{n}}$ denote the interior of $%
W_{n}$ with respect to $\tau $.\ We claim $\left\{ \overset{\circ }{W}%
_{k}x_{m}^{-1}x_{n}\overset{\circ }{W}_{k}:~m,n,k\in \mathbb{N}\right\} $ is
a countable base for $\tilde{\tau}=\tau _{\mathcal{W}}$. \ Let us prove that
for every $n$, $m$, $k\in \mathbb{N}$ and $x\in G$, $\overset{\circ }{W}_{k}x%
\overset{\circ }{W}_{k}$is an open set with respect to $\tau _{\mathcal{W}}$%
. \ Let $s\in \overset{\circ }{W}_{k}\cap G_{r\left( x\right) }$and $t\in 
\overset{\circ }{W}_{k}\cap G^{d\left( x\right) }$. Since $r\left( s\right)
s\in \overset{\circ }{W}_{k}$and $\left( G,\tau \right) $ is a topological
groupoid, there is an open neighborhood $V$ of $r\left( s\right) $ (with
respect to $\tau $) such that $Vs\subset \overset{\circ }{W}_{k}$. The fact
that $\mathcal{W}=\left\{ W_{n}\right\} _{n\in \mathbb{N}}$ is compatible
with the topology of the $r$-fibres implies that there is $p\in \mathbb{N}$
such that $W_{p}\cap G^{r\left( s\right) }\subset V^{-1}\cap G^{r\left(
s\right) }$ and consequently, $W_{p}\cap G_{r\left( s\right) }\subset V\cap
G_{r\left( s\right) }$. \ Hence 
\begin{equation*}
W_{p}s\subset \left( W_{p}\cap G_{r\left( s\right) }\right) s\subset
Vs\subset \overset{\circ }{W}_{k}\text{.}
\end{equation*}%
Similarly, there $q\in \mathbb{N}$ such that $tW_{q}\subset \overset{\circ }{%
W}_{k}$. If $r\in \mathbb{N}$ is such that $W_{r}\subset W_{p}\cap W_{q}$,
then 
\begin{equation*}
sxt\in \overset{\circ }{W}_{r}sxt\overset{\circ }{W}_{r}\subset \overset{%
\circ }{W}_{k}x\overset{\circ }{W}_{k}\text{.}
\end{equation*}

Thus $\overset{\circ }{W}_{k}x\overset{\circ }{W}_{k}$is an open set with
respect to $\tau _{\mathcal{W}}$. \ 

Let us prove that $\left\{ \overset{\circ }{W}_{k}x_{m}^{-1}x_{n}\overset{%
\circ }{W}_{k}:~m,n,k\in \mathbb{N}\right\} $ is a base for $\tilde{\tau}%
=\tau _{\mathcal{W}}$. Indeed, let $x\in G$ and let $A$ be an open subset of 
$G$ with respect to $\tau _{\mathcal{W}}$ such that $x\in A$. Then there is $%
k\in \mathbb{N}$ such that $x\in W_{k}xW_{k}\subset A$. Let $r\in \mathbb{N}$
such that $W_{r}W_{r}\subset W_{k\,}$. Proposition \ref{dens} implies that $%
\left\{ x_{m}^{-1}x_{n}:~m,n\in \mathbb{N}\right\} $ is dense in $G$ with
respect to $\tau _{\mathcal{W}}$. Hence there are $m,n\in \mathbb{N}$ such
that $x_{m}^{-1}x_{n}\in \overset{\circ }{W}_{r}x\overset{\circ }{W}_{r}$,
or equivalently, $x\in \overset{\circ }{W}_{r}x_{m}^{-1}x_{n}\overset{\circ }%
{W}_{r}$. Therefore 
\begin{equation*}
x\in \overset{\circ }{W}_{r}x_{m}^{-1}x_{n}\overset{\circ }{W}_{r}\subset
W_{r}x_{m}^{-1}x_{n}W_{r}\subset W_{r}\overset{\circ }{W}_{r}x\overset{\circ 
}{W}_{r}W_{r}\subset W_{k}xW_{k}\subset A\text{.}
\end{equation*}

$5.$ \ Since the topology $\tilde{\tau}$ is finer than $\tau $, it suffices
to prove that each open set with respect to $\tilde{\tau}$ belong to the
Borel structure generated by $\tau $. \ But as we have proved every open set
with respect to $\tilde{\tau}$ is a countable union of sets of the form $%
W_{r}x_{m}^{-1}x_{n}W_{r}$ which are compact with respect to $\tau $
(because $W_{r}$ is diagonally compact).
\end{proof}

\end{document}